\documentclass[a4paper, 11pt]{article}
\usepackage{mathrsfs}
\usepackage{amsmath,amssymb,amsfonts,amsthm}
\usepackage[T1]{fontenc}
\usepackage[latin9]{inputenc}
\usepackage{graphicx,epsfig, subfigure}
\usepackage{array}
\usepackage{cases}
\usepackage{enumerate,todonotes}
\usepackage{thmtools}
\usepackage{thm-restate}
\usepackage{setspace}  
\usepackage{changes}
\usepackage{cite}
\makeatletter
\def\th@plain{%
  \itshape 
}
\makeatother

\makeatletter
\renewenvironment{proof}[1][\proofname]{\par
  \pushQED{\qed}%
  \normalfont \topsep6\p@\@plus6\p@\relax
  \trivlist
  \item[\hskip\labelsep
        \bfseries
    #1\@addpunct{.}]\ignorespaces
}{%
  \popQED\endtrivlist\@endpefalse
}
\makeatother

\newtheorem{thm}{Theorem}[section]

\newtheorem{cor}[thm]{Corollary}
\newtheorem{claim}[thm]{Claim}
\newtheorem{lemma}[thm]{Lemma}
\newtheorem{example}[thm]{Example}

\newtheorem{defn}[thm]{Definition}

\newtheorem{op}[thm]{Open Question}

\newtheorem*{thm-subcubic16}{Theorem~\ref{listthm1}}
\newtheorem*{thm-subcubic23}{Theorem~\ref{listthm2}}
\numberwithin{equation}{section}

\usepackage[varg]{txfonts}
\usepackage[square, numbers, sort&compress]{natbib}
\ifx\pdfoutput\undefined
 \usepackage[dvipdfm,%
  pdfstartview=FitH,
  bookmarks=true,%
  bookmarksnumbered=true,
  bookmarksopen=true,
  plainpages=false,%
  pdfpagelabels,%
  colorlinks=true,
  linkcolor=blue,
  citecolor=blue,%
  urlcolor=black,
  pdfborder=001]{hyperref}
  \AtBeginDvi{}  
\else
 \usepackage[pdftex,%
  pdfstartview=FitH,
  bookmarks=true,%
  bookmarksnumbered=true,
  bookmarksopen=true,
  plainpages=false,%
  pdfpagelabels,%
  colorlinks=true,
  linkcolor=blue,
  citecolor=blue,%
  urlcolor=black,
  pdfborder=001]{hyperref}%
\fi
\usepackage[capitalise]{cleveref}

\usepackage[left]{lineno}

\DeclareMathOperator{\mad}{mad}

\newcommand{\Lceil}{\left\lceil}
\newcommand{\Rceil}{\right\rceil}

\newcommand{\ceil}[1]{{\Lceil{#1}\Rceil}}

\newcommand{\pc}[2]{packing $(1^{#1}, 2^{#2})$-coloring}
\newcommand{\pch}[2]{packing $(1^{#1}, 2^{#2})$-choosable}

\begin{document}
\title{\LARGE  Between proper and square colorings of sparse graphs

}

\author{
Ilkyoo Choi\thanks{Department of Mathematics, Hankuk University of Foreign Studies, Yongin-si, Gyeonggi-do, Republic of Korea.
 E-mail: \texttt {ilkyoo@hufs.ac.kr}
 and  Discrete Mathematics Group, Institute for Basic Science (IBS), Daejeon, Republic of Korea.
 This work was supported by 
 the Hankuk University of Foreign Studies Research Fund, the Institute for Basic Science (IBS-R029-C1), and the National Research Foundation of Korea(NRF) grant funded by the Korea government(MSIT) (RS-2025-23324220).}
\and
Xujun Liu\thanks{Department of Applied Mathematics, Xi'an Jiaotong-Liverpool University, Suzhou, Jiangsu Province, 215123, China, \texttt{xujun.liu@xjtlu.edu.cn}; the research of X. Liu was supported by the National Natural Science Foundation of China under grant No.~12401466 and the Research Development Fund RDF-21-02-066 of Xi'an Jiaotong-Liverpool University.}
}


\maketitle

\begin{abstract}
\baselineskip 0.60cm
An \emph{$i$-independent set} is a set of vertices whose pairwise distance is at least $i+1$. 
A \emph{proper coloring} (resp. a \emph{square coloring}) of a graph is a partition of its vertices into independent (resp. $2$-independent) sets. 
A \emph{\pc{\ell}{k}} of a graph is a partition of its vertices into $\ell$ independent sets and $k$ $2$-independent sets; 
this is an intermediate coloring between proper coloring and square coloring. 
We investigate classes of sparse graphs that have a proper $(\ell+1)$-coloring but no \pc{\ell}{k} for any finite $k$. 

The Four Color Theorem states that every planar graph is packing $(1^4)$-colorable, and Gr\"otzsch's Theorem says every planar graph with girth at least $4$ is packing $(1^3)$-colorable.
However, for every fixed $k$, we construct a planar graph with no \pc{3}{k} and a planar graph with girth $6$ that has no \pc{2}{k}. 
Moreover, for every positive integer $\ell$, we completely determine the minimum girth condition $g(\ell)$ for which every planar graph with girth at least $g(\ell)$ has a \pc{\ell}{f(\ell)} for some finite $f(\ell)$. 
Our results are actually in terms of maximum average degree.

We also study the list version of packing colorings.
We extend two results of Gastineau and Togni by showing every subcubic graph is both \pch{1}{6} and \pch{2}{3}, and our results are sharp.
In addition, we strengthen Voigt's example of a planar graph that is not $4$-choosable by constructing a planar graph that is not \pch{4}{k} for every positive integer $k$.


\vspace{3mm}\noindent \emph{Keywords}: graph coloring; planar graph; square coloring; list coloring; packing coloring; packing list-coloring 
\end{abstract}

\baselineskip 0.60cm

\section{Introduction}


All graphs in this paper are simple, which means no loops and no parallel edges.
For a graph $G$, let $V(G)$ and $E(G)$ denote the vertex set and edge set, respectively, of $G$. 
Given a graph $G$, the graph $G^i$ is defined as $V(G^i) = V(G)$ and $E(G^i) = \{uv \text{ }|\text{ } u,v \in V(G) \text{ and } d_G(u,v) \le i\}$.
A graph is \emph{(sub)cubic} if every vertex has degree (at most) $3$. 
A {\it $k$-vertex} (resp. {\it $k^+$-vertex}, {\it $k^-$-vertex}) is a vertex of degree $k$ (resp. at least $k$, at most $k$).
A {\it $k$-neighbor} of a vertex $v$ is a neighbor of $v$ that is a $k$-vertex; a {\it $k^+$-neighbor} and a {\it $k^-$-neighbor} are defined analogously. 
For a vertex $v$, the {\it closed neighborhood} of $v$, denoted $N[v]$, is $N(v)\cup\{v\}$. 
The \emph{maximum average degree} of a graph $G$, denoted $\mad(G)$, is $\max \left\{\frac{2|E(H)|}{|V(H)|}: H \subseteq G\right\}$.
The \emph{girth} of a graph is the length of a shortest cycle. 
Generalizing the notion of an independent set, an {\it $i$-independent set} is a set of vertices whose pairwise distance is at least $i+1$. 
Note that a $1$-independent set is an independent set.
A {\it proper $k$-coloring} of a graph is a partition of its vertex set into $k$ independent sets. 
A \emph{square $k$-coloring} of a graph is a partition of its vertex set into $k$ $2$-independent sets. 
The minimum $k$ for which a graph $G$ has a square $k$-coloring is denoted $\chi_2(G)$. 

For a sequence $S = (s_1, \ldots, s_k)$ of non-decreasing positive integers, a \emph{packing $S$-coloring} of a graph $G$ is a partition of $V(G)$ into $V_1, \ldots, V_k$ such that each $V_i$ is $s_i$-independent.
We say a vertex in $V_i$ is colored with an {\em $i$-color}, which is usually denoted by $i$ with some subscript as in $i_x$ or $i_y$.   
We use $(s_1^{t_1}, \ldots, s_k^{t_k})$ to denote the tuple where $s_i$ appears with multiplicity $t_i$ for every $i$. 
A packing $S$-coloring is a proper $k$-coloring (resp. square $k$-coloring) when all entries of $S$ are $1$ (resp. 2).
In this paper, we will focus on packing $S$-colorings (and its list version, defined later) where each entry of $S$ is either $1$ or $2$; this is a coloring between a proper coloring and a square coloring. 
We are particularly interested in determining situations where a $1$-color can or cannot be replaced with finitely many $2$-colors.


\subsection{Packing coloring for sparse graphs}

The celebrated Four Color Theorem~\cite{AH1,AHK1,RSST1} states that every planar graph has a proper $4$-coloring. 
Since the closed neighborhood of a $k$-vertex in any graph $G$ forms a complete graph on $k+1$ vertices in $G^2$, there is no constant $C$ for which all planar graphs have a square $C$-coloring. 
The well-known conjecture by Wegner~\cite{W1} from 1977 states that if $G$ is a planar graph with maximum degree $\Delta$, then 
\[
\chi_2(G) \le 
\begin{cases} 
7 & \mbox{ if $\Delta = 3$}\\
\Delta + 5 & \mbox{ if $ \Delta\in\{4,5,6, 7\}$}\\
\left\lfloor \frac{3}{2} \Delta \right\rfloor + 1 & \mbox{ if $\Delta \ge 8$}. \\
\end{cases}
\]
Thomassen~\cite{T2}, and independently Hartke, Jahanbekam, and Thomas~\cite{HJT1} confirmed the above conjecture when $\Delta = 3$. Wegner's conjecture remains open for all $\Delta \ge 4$.
Recently, Bousquet, Deschamps, de Meyer, and Pierron~\cite{BDMP1} proved the current best upper bound of $12$ when $\Delta = 4$. 
Amini, Esperet, and van den Heuvel~\cite{AEH1}, and independently Havet, van den Heuvel, McDiarmid, and Reed~\cite{HHMR1,HHMR2} showed that the conjecture holds asymptotically, namely, $\chi_2(G) \le \frac{3}{2} \Delta + o (\Delta)$ as $\Delta \to \infty$. 

We consider packing colorings of planar graphs with no restrictions on vertex degrees. 
We initiate the study of given an integer $\ell$, determining the minimum girth condition $g=g(\ell)$ for which there exists a finite $f_g(\ell)$ where every planar graph with girth at least $g$ has a \pc{\ell}{f_g(\ell)}. 
We succeed in determining $g(\ell)$ for all positive integers $\ell$. 
It would be interesting to determine the minimum value of $f_g(\ell)$ as well. 
The Four Color Theorem states that every planar graph is packing $(1^4)$-colorable, so $g(4)=3$ and $f_3(\ell)=0$ for all $\ell\geq 4$. 
Gr\"otzsch~\cite{G1} proved that every planar graph with girth at least $4$ is packing $(1^3)$-colorable, so $g(3)\leq 4$. 
However, for every fixed $k$, we construct a planar graph with no \pc{3}{k}, so $g(3)=4$. 
(See \Cref{ex:general}.)

For $\ell=2$, we construct a planar graph with girth $6$ that has no \pc{2}{k} for every finite $k$, so $g(2)>6$.  
(See \Cref{girth6}.)
One of our main results implies that every planar graph with girth at least $7$ has a \pc{2}{12}, so $g(2)=7$ and $f_7(2)\leq 12$.
In fact, we prove a stronger result in terms of maximum average degree. 
Note that Euler's formula implies that a planar graph with girth at least $g$ has maximum average degree less than $\frac{2g}{g-2}$.

\begin{thm}\label{thm:main-theorem-1}
Every graph with maximum average degree less than $\frac{14}{5}$ has a \pc{2}{12}. \end{thm}

\begin{cor}\label{girth7}
Every planar graph with girth at least $7$ has a \pc{2}{12}.
\end{cor}

Note that the graph $F(1, k)$ in \Cref{ex:general} is a tree (so it is planar) that does not have a \pc{1}{k} for any $k$, so $g(1)$ does not exist.

As $f_7(2)\leq 12$, it would be interesting to determine the exact girth condition for a planar graph to have a \pc{2}{k} for all values of $k\leq 11$.

\begin{op}\label{op:gk}
For $k\in\{1,\ldots, 11\}$, what is the minimum positive integer $g_k$ such that every planar graph with girth at least $g_k$ has a \pc{2}{k}?
\end{op}

When $k=1$, Brandt, Ferrara, Kumbhat, Loeb, Stolee, and Yancey~\cite{BFKLSY1} proved a result that implies every graph $G$ with $\mad(G)<\frac{5}{2}$ has a \pc{2}{1}.
The bound on maximum average degree is tight, and their tightness example is also a graph that does not have a \pc{2}{1}.
Their result implies that planar graphs with girth at least $10$ have a \pc{2}{1}.
On the other hand, we were able to construct a planar graph with girth $7$ that has no \pc{2}{1}. (See Example~\ref{sparse}.)
Therefore $g_1\in\{8,9,10\}$. 

When $k=2$, Choi and Park~\cite{CP1} proved a result that implies every graph $G$ with $\mad(G) \le \frac{8}{3}$ has a \pc{2}{2}. As a corollary of their result, every planar graph with girth at least $8$ has a \pc{2}{2}.
Recall that we found in Example~\ref{girth6}, for each positive integer $k$, a planar graph with girth $6$ that has no \pc{2}{k}. 
Therefore, for $k\in\{2, \ldots, 11\}$, $g_k\in\{7,8\}$.

\subsection{Packing choosability for sparse graphs}


A {\em list assignment} of a graph $G$ assigns a list $L(v)$ of available colors to each vertex $v$ of $G$. 
An {\em $L$-coloring} is an assignment of colors from $L(v)$ to each vertex $v \in V(G)$ such that adjacent vertices receive distinct colors. 
A graph $G$ is {\em $k$-choosable} if an $L$-coloring always exists provided $|L(v)| \ge k$ for each vertex $v \in V(G)$. 
The {\em list chromatic number} of a graph $G$, denoted $\chi_{\ell}(G)$, is the minimum $k$ such that $G$ is $k$-choosable. 
This notion was introduced by Vizing~\cite{V1}, and independently by Erd\H{o}s, Rubin, and Taylor~\cite{ERT1}.
Ever since, extending a chromatic parameter to its list version has been a widely pursued avenue of research in the literature. 
In this vein, we define the list version of packing $S$-coloring and define packing $S$-choosable graphs. 



\begin{defn}
Let $\mathcal L_s$ be a set of $s$ pairwise disjoint sets $L_1, \ldots, L_s$ (of colors). 
For a graph $G$, a \emph{packing $\mathcal L_s$-list assignment} $L$ assigns to each vertex $v \in V(G)$ a list $L(v)\subseteq \bigcup_{L_i\in\mathcal L_s} L_i$ of available colors for $v$. 
A \emph{packing $L$-coloring} is a function $\varphi$ such that 

(1) $\varphi(v) \in L(v)$ for each $v \in V(G)$, and 

(2) the set of vertices $v$ with $\varphi(v) \in L_i$ is $i$-independent. 
\end{defn}

\begin{defn}
A graph $G$ is \emph{packing $(1^{t_1}, \ldots, s^{t_s})$-choosable} if it has a packing $L$-coloring for every $\mathcal L_s$-list assignment $L$ where $|L(v)\cap L_i| \ge t_i$ for all $v \in V(G)$ and every $i\in\{1, \ldots, s\}$. 
\end{defn}


We first discuss subcubic graphs. 
Brooks's Theorem~\cite{1941Brooks} states that every subcubic graph has a proper $3$-coloring except $K_4$. 
Cranston and Kim~\cite{CK1} proved that every  subcubic graph has a square $8$-coloring except the Petersen graph.  
Gastineau and Togni~\cite{GT1} showed that every subcubic graph has both a \pc{1}{6} and a \pc{2}{3}. 
Both results are sharp since the Petersen graph has neither a \pc{1}{5} nor a \pc{2}{2}.
Furthermore, they posed the question of whether every subcubic graph except the Petersen graph has a \pc{1}{5}.
(See~\cite{LW1,LZZ1} for related results in this direction.)
%
We extend their results by showing that every subcubic graph is \pch{1}{6} and \pch{2}{3}. Our results are sharp by the Petersen graph.
Furthermore, our proof methods are different from those of Gastineau and Togni~\cite{GT1}.

\begin{thm-subcubic16}
Every subcubic graph is \pch{1}{6}.    
\end{thm-subcubic16}

\begin{thm-subcubic23}
Every subcubic graph is \pch{2}{3}.    
\end{thm-subcubic23}

We also investigate packing choosability for planar graphs.
In contrast to the Four Color Theorem, Vogit~\cite{V2} showed that there exist planar graphs that are not $4$-choosable, which is equivalent to packing $(1^4)$-choosable. 
We strengthen this result by, for each $k$, constructing a planar graph that is not \pch{4}{k} (see Example~\ref{listexample}); namely, the addition of an arbitrary number of $2$-colors does not ensure a coloring. 
Recall that Thomassen~\cite{T1} in 1994 showed via a remarkable proof that every planar graph is $5$-choosable. 

\bigskip

In Section~\ref{sec:ex}, we first present all of our examples mentioned in the introduction. 
We then prove Theorem~\ref{thm:main-theorem-1} in Section~\ref{proofmaintheorem}, which is split into two subsections. 
In Section~\ref{reduceible_configuration}, we show that certain configurations cannot appear in a minimum counterexample. 
We complete the proof in Section~\ref{discharging} via the discharging method. 
 Theorem~\ref{listthm1} and Theorem~\ref{listthm2} are proven in Section~\ref{list}.
 We end the paper with some open questions in the last section.

\section{Examples}\label{sec:ex}

We present three examples of graphs that have a proper $(\ell+1)$-coloring but do not have a \pc{\ell}{k} for any finite $k$. 
The fourth example is a planar graph that is not \pch{4}{k} for any finite $k$. 

 \begin{figure}[ht]
 \begin{center}
   \includegraphics[scale=0.18]{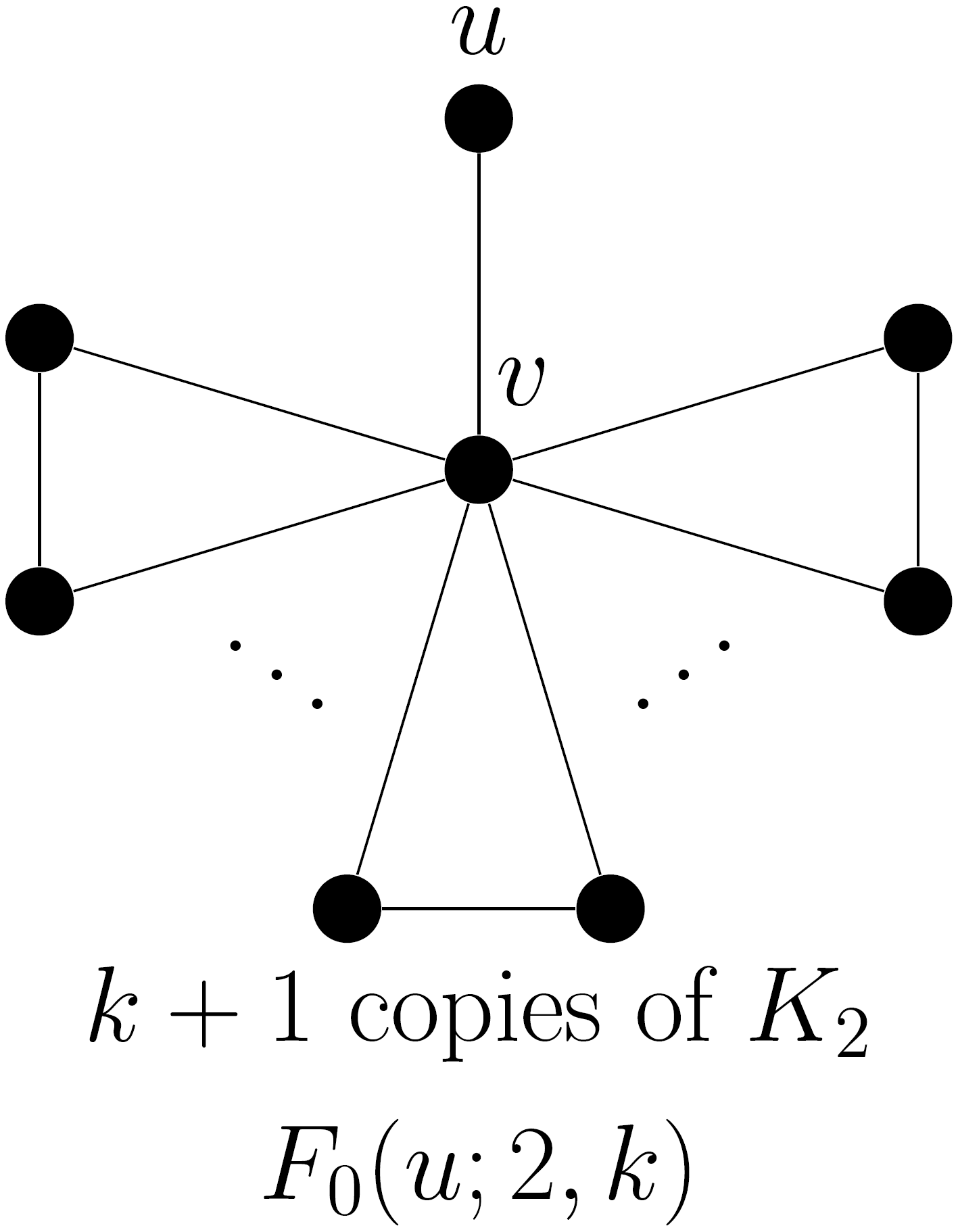} \hspace{8mm}
   \includegraphics[scale=0.16]{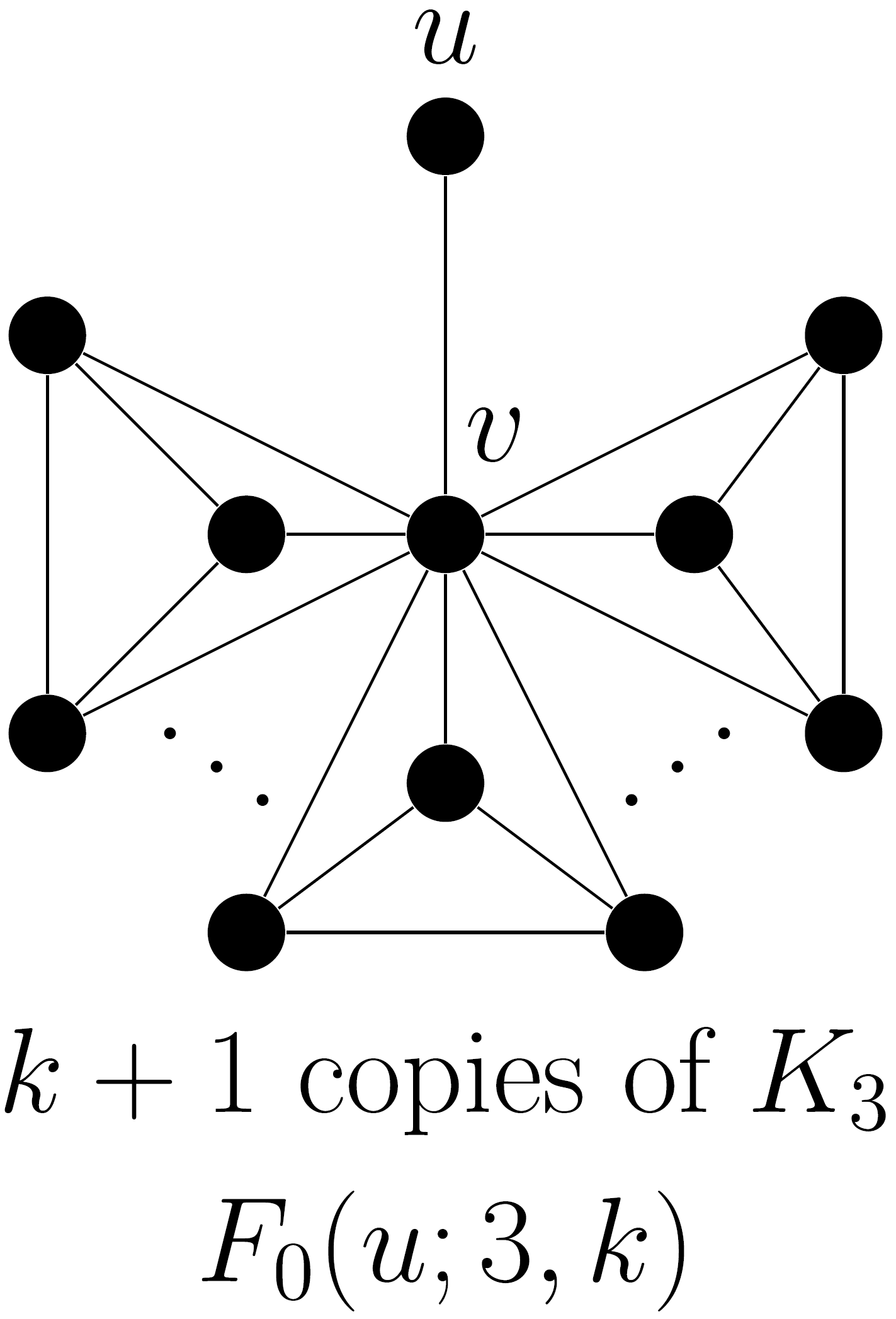} \hspace{10mm}
   \includegraphics[scale=0.6]{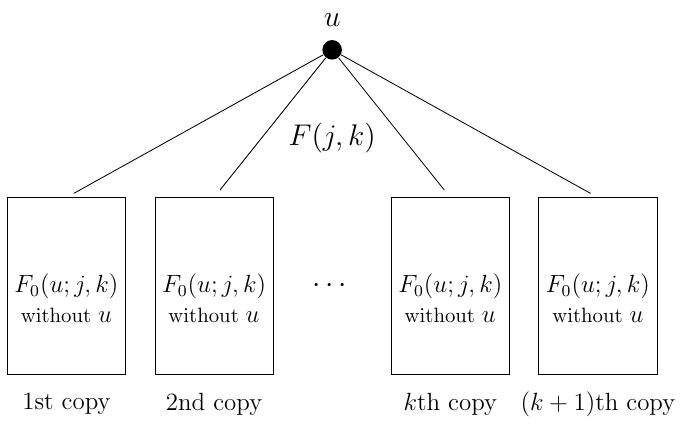} 
\caption{Graphs with a proper $(j+1)$-coloring but no \pc{j}{k}.}\label{fig:example_1}
\end{center}
\vspace{-8mm}
\end{figure}

\begin{example}\label{ex:general}
For positive integers $j, k$, let $F_0(u; j, k)$ be the join of $K_1$ and $(k+1)K_{j}\cup K_1$ where $u$ is a $1$-vertex.
Obtain the graph $F(j, k)$ by starting with $k+1$ copies of $F_0(u; j, k)$ and identifying all copies of $u$ into a single vertex. See Figure~\ref{fig:example_1}.
\end{example}

A graph with a \pc{\ell}{k} is not guaranteed to have a \pc{\ell-1}{f(k)} for any function $f(k)$, even when restricted to bipartite graphs, outerplanar graphs, and planar graphs. 
Indeed, the graph $F(\ell, k)$ in \Cref{ex:general} has a packing $(1^{\ell+1})$-coloring, yet it does not have a \pc{\ell}{k} for any finite $k$.
To see this, suppose $F(\ell, k)$ has a \pc{\ell}{k}. 
Then the vertex $v$ of maximum degree in each copy of $F_0(u; \ell, k)$ must receive a $2$-color; 
if $v$ gets a $1$-color, then each copy of $K_{\ell+1}$ must have a $2$-colored vertex $w$, but there are $k+1$ such $w$'s in the neighborhood of $v$, which is a contradiction. 
Now, all $k+1$ neighbors of $u$ are $2$-colored, which is a contradiction.
Hence, $F(\ell, k)$ has no \pc{\ell}{k}. 
Note that $F(1, k)$ is a tree, $F(2, k)$ is outerplanar, and $F(3, k)$ is planar. 

\begin{figure}[ht]
\begin{center}
  \includegraphics[scale=0.16]{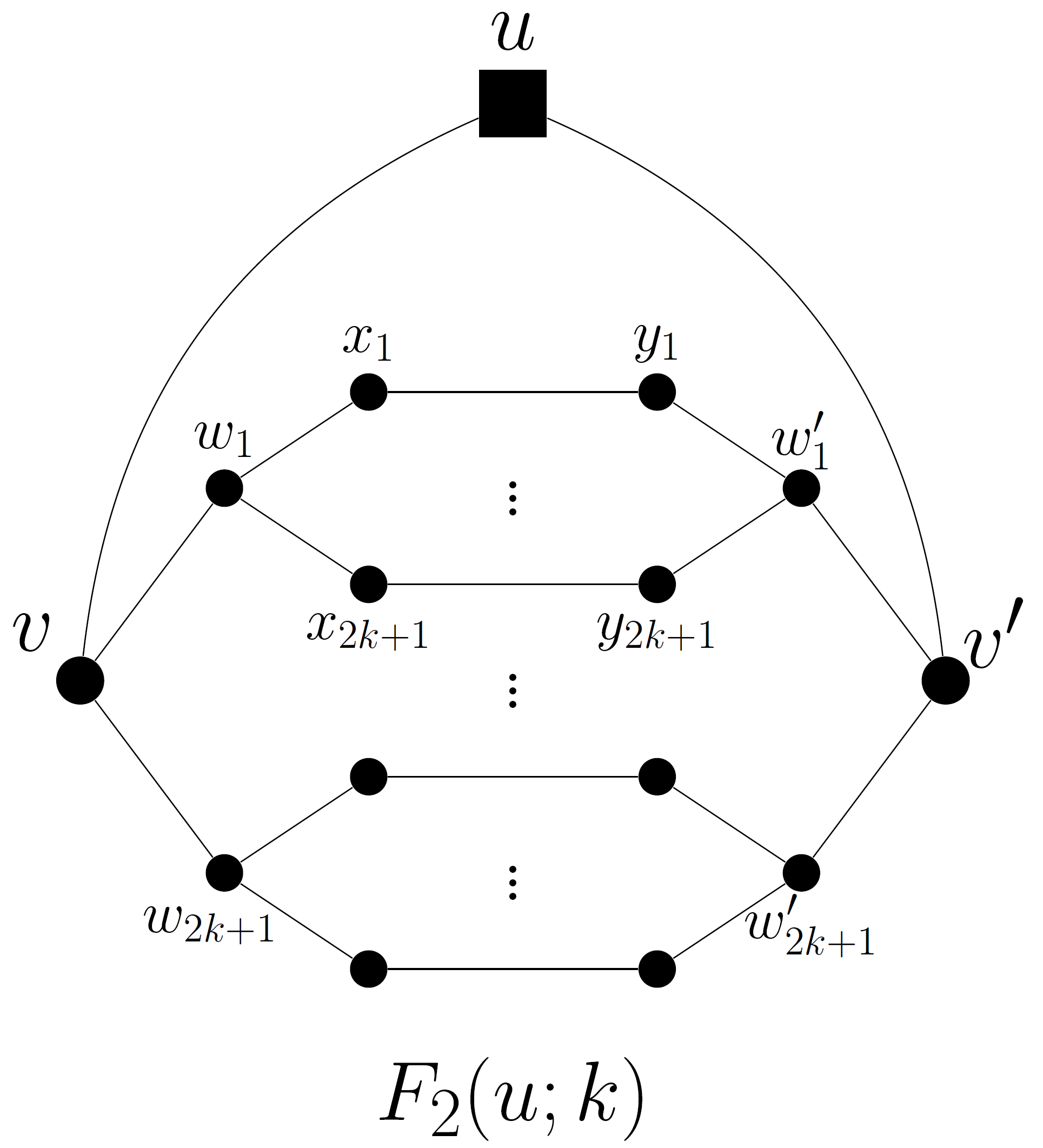} \hspace{5mm}
  \includegraphics[scale=0.4]{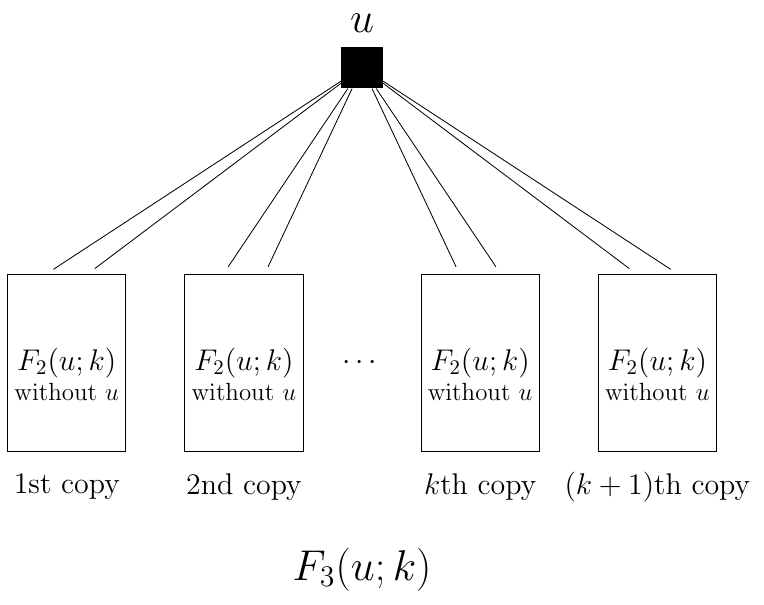} \hspace{5mm}
  \includegraphics[scale=0.4]{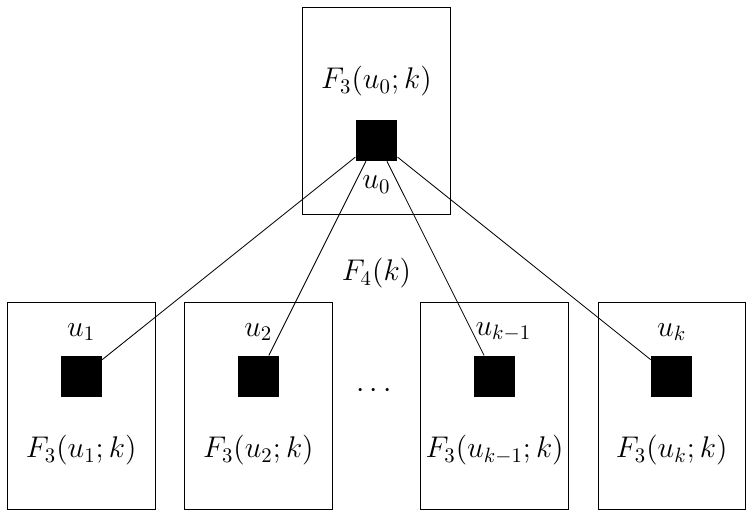} 
\caption{A planar graph with girth $6$ that has no \pc{2}{k}.}\label{fig:example_5}
\end{center}
\vspace{-8mm}
\end{figure}

\begin{example}\label{girth6}
    For a positive integer $k$, let $F_1(w, w'; k)$ be the graph obtained by connecting two vertices $w$ and $w'$ with $2k+1$ internally vertex disjoint paths of length $3$. 
    Obtain the graph $F_2(u; k)$ by starting with $2k+1$ copies of $F_1(w, w'; k)$ and a disjoint path $vuv'$, and adding $vw$ and $v'w'$ for $w, w'$ in each copy of $F_1(w, w'; k)$. 
    Obtain the graph $F_3(u; k)$ by starting with $k+1$ copies of $F_2(u; k)$ and identifying all copies of $u$ into a single vertex.
    Finally, obtain the graph $F_4(k)$ by starting with $F_3(u_0; k)$ and $k$ copies of $F_3(u; k)$, and adding $u_0u$ for $u$ in each copy of $F_3(u; k)$. 
    See \Cref{fig:example_5}.
\end{example}


It is easy to check that the graph $F_4(k)$ is planar and has girth $6$, so $F_4(k)$ has a proper $3$-coloring.  
We claim that $F_4(k)$ has no \pc{2}{k} for any finite $k$.
Suppose $F_4(k)$ has a \pc{2}{k} $\varphi$ with $1$-colors $1_a, 1_b$. 
Since the $k+1$ vertices corresponding to $u$ in each copy of $F_3(u; k)$ have pairwise distance at most $2$, a vertex among these vertices must be colored with a $1$-color, say $\varphi(u_0)=1_a$.
Since the $2k+2$ neighbors of $u_0$ in $F_3(u_0; k)$ have pairwise distance at most $2$, $k+2$ vertices among these neighbors must be colored with a $1$-color. 
Moreover, two of these vertices $v_0, v'_0$ must be in the same copy $Y$ of $F_2(u_0;k)$. 
Since $\varphi(u_0)=1_a$, $\varphi(v_0)=\varphi(v'_0)=1_b$.
Since the $2k+1$ neighbors of $v_0$ (resp. $v'_0$) excluding $u_0$ in $Y$ have pairwise distance at most $2$, $k+1$ vertices among these neighbors must be colored with a $1$-color. 
Therefore, $Y$ has a subgraph $F_1(w_0, w_0';k)$ for which $\varphi(w_0)=\varphi(w'_0)=1_a$.
Now, the $2k+1$ neighbors of $w_0$ (resp. $w'_0$) excluding $v_0$ (resp. $v'_0$) in $F_1(w_0, w_0';k)$ have pairwise distance at most $2$, so $k+1$ vertices among these neighbors must be colored with a $1$-color. 
Hence, there is a path of length $3$ from $w_0$ to $w'_0$ where the two internal vertices receive the same $1$-color, which is a contradiction. 
Hence, $F_4(k)$ has no \pc{2}{k}.

\def\xspace{1}
\def\yspace{0.75}

\def\figOneGadget{
\draw(0,0) node[blacknode](c){};
\draw (c)++(0,\yspace) node[label=above:$v_1$](){};
\draw (c)++(0,-\yspace) node[label=below:$v_2$](){};
\draw(-\xspace*2,0) node[blacknode,label=left:$x$](l){};
\draw(\xspace*2,0) node[blacknode,label=right:$y$](r){};
\foreach \i in {1,0,-1,-2}
    \draw (c)++(\xspace*\i+\xspace*0.5, \yspace*2) node[blacknode](u2\i){}
    (c)++(\xspace*\i+\xspace*0.5, -\yspace*2) node[blacknode](d2\i){};
\foreach \i in {0,1,-1}
    \draw (c)++(\xspace*\i, \yspace*1) node[blacknode](u1\i){}
    (c)++(\xspace*\i, -\yspace*1) node[blacknode](d1\i){};

\draw (d10)--(c)--(u10);
 \foreach \c in {u,d}
    \draw (l)--(\c2-2)--(\c2-1)--(\c20)--(\c21)--(r)
    (\c2-2)--(\c1-1)--(\c10)--(\c11)--(\c21);

\draw (l)++(0,-3) node[graynode,label=$x$](ll){};
\draw (r)++(0,-3) node[graynode,label=$y$](rr){};
\draw[thick,dotted] (ll) to  (rr);



}

\def\figOneGadgett{

\draw (0,0) node [graynode,label=above:$z$] (c){};
 \foreach \i in {1,...,9}
     \draw (c)++(-90+360/9*\i:1.5) node[graynode] (u\i){};
 \foreach \x [count=\xi from 2] in {1,...,8}                  \draw  (u\x)--(u\xi);
\draw  (u9)--(u1);

\foreach \c in {1,...,9}
    \draw[thick,dotted] (c) to (u\c);
\draw (c)++(0,-3) node[graynode,label=$z$](n){};
\draw (n) to [out=30, in=180-30,looseness=30](n) ;
}

\def\figOneEx{

\draw (0,0) node (c){};
 \foreach \i in {1,...,9}{
    \draw (c)++(-90+360/9*\i:1.5) node[graynode] (u\i){};
    \draw (u\i) to [out=-90+360/9*\i+45, in=-90+360/9*\i-45,looseness=30](u\i) ;
}
 \foreach \x [count=\xi from 2] in {1,...,8}                  \draw  (u\x)--(u\xi);
\draw  (u9)--(u1);

}

\begin{figure}

\begin{tikzpicture}
[scale=1,auto=left, 
blacknode/.style={circle,draw,fill=black,minimum size = 6pt,inner sep=0pt}, 
graynode/.style={circle,draw,fill=gray!30,minimum size = 6pt,inner sep=0pt}
]
\def\spacee{2}
 \begin{scope}[xshift=\spacee*0cm]\figOneGadget\end{scope}
 \begin{scope}[xshift=\spacee*2.75cm]\figOneGadgett\end{scope}
 \begin{scope}[xshift=\spacee*5.5cm]\figOneEx\end{scope}
\end{tikzpicture}

\caption{
A planar graph with girth $7$ that has no \pc{2}{1}.}
\label{fig:pc21-girth7}
\end{figure}

\begin{example}\label{sparse}
    Let $H^1_1(x,y)$ be the left graph in Figure~\ref{fig:pc21-girth7}.
    Obtain the graph $H^2_1(C,z)$ from an odd $7^+$-cycle $C$ and a vertex $z$ not on $C$ by adding a copy of $H^1_1(z, v)$ for each vertex $v$ on $C$. 
    Obtain the graph $H_1(C')$ from an odd $7^+$-cycle $C'$ and adding a copy of $H^2_1(C, v)$ for each vertex $v$ on $C'$. 
    See Figure~\ref{fig:pc21-girth7}. 
\end{example}

Let $H_1=H_1(C')$ for an odd $7^+$-cycle $C'$. 
It is easy to check that $H_1$ is planar and has girth $7$.
We claim that $H_1$ has no \pc{2}{1}.
Note that the distance between $x$ and $y$ in $H^1_1(x,y)$ is $5$. 
Suppose $H_1$ has a \pc{2}{1} $\varphi$.
Since $C'$ is an odd cycle, at least one vertex $v'$ on $C'$ received a $2$-color by $\varphi$.
Since $C$ is also an odd cycle in $H^2_1(C,v')$, at least one vertex $v$ on $C$ received a $2$-color by $\varphi$. 
Now, in $H^1_1(v,v')$, it must be that both $v_1$ and $v_2$ received $2$-colors by $\varphi$, which is a contradiction. 
Hence, $H_1$ has no \pc{2}{1}.  


\begin{figure}
\begin{center}
  \includegraphics[scale=0.7]{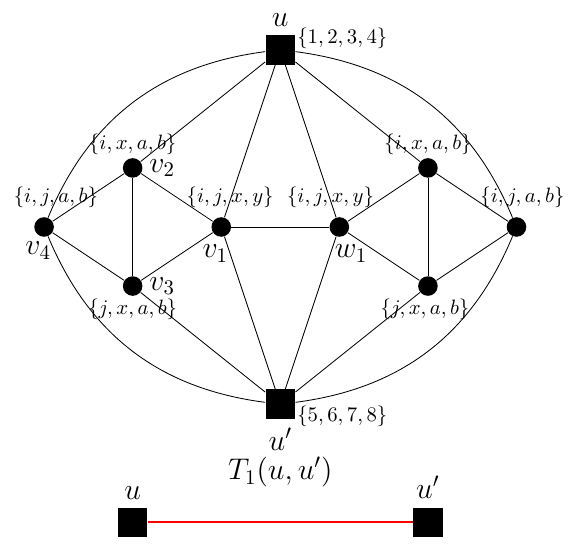} \hspace{0mm}
  \includegraphics[scale=0.72]{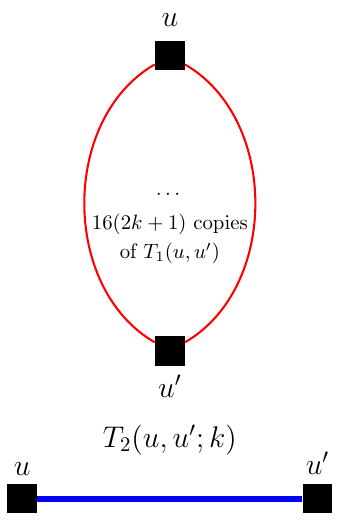} \hspace{5mm}
  \includegraphics[scale=0.6]{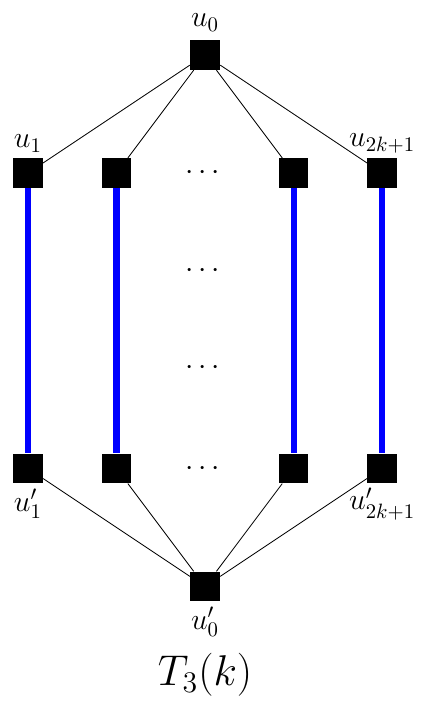} 
\end{center}
\caption{A planar graph that is not \pch{4}{k}}\label{list-1}
\end{figure}

\begin{example}\label{listexample}
Let $T_1(u, u')$ be the left graph in Figure~\ref{list-1}. 
Obtain the graph $T_2(u, u'; k)$ by starting with $16(2k+1)$ copies of $T_1(u, u')$ and identifying all copies of $u$ into a single vertex and identifying all copies of $u'$ into a single vertex.
Obtain the graph $T_3(k)$ by starting with $2k+1$ copies of $T_2(u, u'; k)$ and adding two vertices $u_0$ and $u_0'$ where $u_0$ is  adjacent to all copies of $u$ in $T_2(u, u'; k)$ and $u_0'$ is adjacent to all copies of $u'$ in $T_2(u, u'; k)$. 
See Figure~\ref{list-1}. 
\end{example}

It is easy to check that $T_3(k)$ is a planar graph.
We claim that $T_3(k)$ is not \pch{4}{k}.
In other words, there is a list assignment $L$ with $|L_1(v)| = 4, |L_2(v)| = k$ for every $v \in V(T_3(k))$ such that $T_3(k)$ has no packing $L$-coloring.
Consider the following list assignment $L$ on $T_3(k)$. 
First, every vertex has the same list of $2$-colors. 
We now define the list of $1$-colors for each vertex. 
For a vertex $v\in\{u_0, u'_0\}$, let $L_1(v)$ be an arbitrary set of four colors. 
Note that each vertex $v\not\in\{u_0, u'_0\}$ of $T_3(k)$ is part of a copy of $T_2(u, u';k)$.
Fix an arbitrary copy of $T_2(u, u';k)$.
Let $L_1(u)=\{1,2,3,4\}$ and $L_1(u')=\{5,6,7,8\}$.
For $v\not\in\{u, u'\}$, let $L_1(v)$ be as in the left graph in Figure~\ref{list-1} where for each $(\alpha, \beta)\in \{1,2,3,4\}\times\{5,6,7,8\}$, $(i, j)=(\alpha, \beta)$ holds in exactly $2k+1$ copies of $T_1(u, u')$ within this copy of $T_2(u, u';k)$.
In general, a copy of $T_1(u, u')$ is \emph{$(\alpha, \beta)$-bad} if the corresponding $(i, j)=(\alpha,\beta)$. 
We remark that in an $L$-coloring $\varphi$ of an $(\alpha, \beta)$-bad copy of $T_1(u, u')$, if $(\varphi(u), \varphi(u'))=(\alpha, \beta)$, then there is a vertex $z\not\in\{u, u'\}$ of $T_1(u, u')$ where $\varphi(z)$ is a $2$-color.


Suppose $T_3(k)$ has a packing $L$-coloring $\varphi$.
Since all neighbors of $u_0$ (resp. $u'_0$) have pairwise distance $2$, at least $k+1$ neighbors of $u_0$ (resp. $u'_0$) must be colored with a $1$-color. 
Therefore, there must be a copy of $T_2(u, u'; k)$ for which both $\varphi(u)$ and $\varphi(u')$ are $1$-colors. 
In particular, there are $2k+1$ copies of $T_1(u, u')$ within that copy of $T_2(u, u'; k)$ that are $(\varphi(u), \varphi(u'))$-bad. 
In each such $(\varphi(u), \varphi(u'))$-bad  $T_1(u, u')$, there is a vertex $z\not\in\{u, u'\}$ for which $\varphi(z)$ is a $2$-color. 
Note that $z$ is a neighbor of $u$ or $u'$. 
Since all neighbors of $u$ (resp. $u'$) have pairwise distance $2$, $u$ (resp. $u'$) has at most $k$ neighbors for which $\varphi(z)$ is a $2$-color. 
Therefore, there is an $(\varphi(u), \varphi(u'))$-bad $T_1(u, u')$ where all vertices received a $1$-color by $\varphi$, which is a contradiction. 




\section{Proof of Theorem~\ref{thm:main-theorem-1}}\label{proofmaintheorem}


Let $G_k$ be a graph with no \pc{2}{k}, but every proper subgraph has a \pc{2}{k}.
Assume the $1$-colors are $1_a, 1_b$ and the $2$-colors are $2_1, \ldots, 2_k$.

\subsection{Reducible Configurations}\label{reduceible_configuration}

\begin{lemma}\label{lem:1vx}
No vertex of $G_k$ is a $1$-vertex. 
\end{lemma}
\begin{proof}
Suppose $G_k$ has a $1$-vertex $u$, whose unique neighbor is $v$. 
By the minimality of $G_k$, the graph $G_k-u$ has a \pc{2}{k} $\varphi$. 
Extend $\varphi$ to $u$ by coloring $u$ with a  $1$-color that is not $\varphi(v)$. 
This is a \pc{2}{k} of all of $G_k$, which is a contradiction.
\end{proof}

A {\it $t$-thread} is a path on $t$ $2$-vertices. 

\begin{lemma}\label{lem:2thread}
If $u_1u_2$ is a $2$-thread in $G_k$ where $v_i$ is the neighbor of $u_i$ that is not $u_{3-i}$, then both $v_1$ and $v_2$ are $(k+2)^+$-vertices. 
\end{lemma}
\begin{proof}
    Let $u_1u_2$ be a $2$-thread of $G_k$ and let $v_i$ be the neighbor of $u_i$ that is not $u_{3-i}$.
    By the minimality of $G_k$, the graph $G_k-\{u_1, u_2\}$ has a \pc{2}{k} $\varphi$. 
    If $\varphi(v_1)$ is not a $1$-color, then we can extend $\varphi$ to a \pc{2}{k} of all of $G_k$ by greedily coloring $u_2, u_1$ with $1$-colors,
    a contradiction.
    Thus, by symmetry, we may assume both $\varphi(v_1)$ and $\varphi(v_2)$ are $1$-colors. 
    If $\varphi(v_1)\neq \varphi(v_2)$, then we can extend $\varphi$ to a \pc{2}{k} of all of $G_k$ by setting $\varphi(u_1)=\varphi(v_2)$ and $\varphi(u_2)=\varphi(v_1)$,
    a contradiction.
    Thus, $\varphi(v_1)=\varphi(v_2)$, and without loss of generality assume $\varphi(v_1)=1_a$. 
    
    Assume $v_1$ is a $(k+1)^-$-vertex. 
    If $v_1 = v_2$, then there is a $2$-color $2_x\not\in\varphi(N(v_1)-u_1-u_2)$, so we can extend $\varphi$ to a \pc{2}{k} of all of $G_k$ by setting $\varphi(u_1)=1_b$ and $\varphi(u_2)=2_x$, a contradiction. 
    Thus, $v_1 \neq v_2$.
    We must not be able to recolor $v_1$ with $1_b$, so $v_1$ has a neighbor $w$ colored $1_b$. 
    Since $v_1$ has at most $k-1$ neighbors excluding $u_1$ and $w$, there is a $2$-color $2_x$ we can use on $u_1$. 
    Now, extend $\varphi$ to a \pc{2}{k} of all of $G_k$ by coloring $u_1$ with $2_x$ and $u_2$ with $1_b$, a contradiction.    
    Therefore, $v_1$ is a $(k+2)^+$-vertex, and by symmetry $v_2$ is also a $(k+2)^+$-vertex.
\end{proof}

\begin{lemma}\label{lem:all2nb}
If a vertex $u$ of $G_k$ does not have a $3^+$-neighbor, then $u$ is a $(k+2)^+$-vertex. 
\end{lemma}
\begin{proof}
Suppose $G_k$ has a $(k+1)^-$-vertex $u$ with no $3^+$-neighbors.
By \Cref{lem:1vx}, the neighbors $u_1, \ldots, u_d$ of $u$ are all $2$-vertices. 
Note that by \Cref{lem:2thread}, no neighbors of $u$ are adjacent to each other. 

%

Let $v_i$ be the neighbor of $u_i$ that is not $u$. 
By the minimality of $G_k$, the graph $G_k - \{u, u_1, \ldots, u_d\}$ has a \pc{2}{k} $\varphi$. 
If a $2$-color $2_x$ does not appear on $v_1, \ldots, v_d$, then extend $\varphi$ to a \pc{2}{k} of all of $G_k$ by coloring $u$ with $2_x$ and coloring each $u_i$ with a $1$-color that is not $\varphi(v_i)$, a contradiction. 
If a $1$-color $1_x$ does not appear on $v_1, \ldots, v_d$, then extend $\varphi$ to a \pc{2}{k} of all of $G_k$ by coloring $u$ with $1_y$ and coloring each $u_i$ with $1_x$ where $x\neq y$, a contradiction. 
Therefore, all $1$-colors and $2$-colors must appear in $v_1, \ldots, v_d$, so $d\geq k+2$. 
\end{proof}

By \Cref{lem:all2nb}, we have the following corollary.

\begin{cor}\label{cor:no3thread}
There is no $3$-thread in $G_k$. 
\end{cor}




\begin{lemma}\label{lem:only2threads}
A vertex of $G_k$ cannot be adjacent to only $2$-threads. 
\end{lemma}
\begin{proof}
Suppose $G_k$ has a $d$-vertex $u$ that is adjacent to $d$ $2$-threads $v_1v_1', \ldots, v_dv'_d$ where $v_1, \ldots, v_d$ are neighbors of $u$.
By the minimality of $G_k$, the graph $G_k - \{v_1,v_1',\ldots, v_d,v_d'\}$ has a \pc{2}{k} $\varphi$. 
We can extend $\varphi$ to a \pc{2}{k} of all of $G_k$ by coloring $u$ with a $2$-color and greedily coloring $v_1',v_1,v_2',v_2, \ldots, v_d', v_d$ with $1$-colors,  a contradiction.
\end{proof}

\begin{lemma}\label{lem:two2nbs}
    Let $u$ be a $3$-vertex of $G_k$ with neighbors $u_1, u_2, v_3$. 
    If $u_1, u_2$ are $2$-vertices where $v_i$ is the neighbor of $u_i$ that is not $u$, then $v_3$ is a $(k+1)^+$-vertex and at least one of $v_1,v_2,v_3$ is a $(k+2)^+$-vertex. 
    
\end{lemma}    


\begin{proof}
By the minimality of $G_k$, the graph $G_k-\{u, u_1, u_2\}$ has a \pc{2}{k} $\varphi$. 
We first consider the case where $u_1$ is adjacent to $u_2$. 
By \Cref{lem:2thread}, $u$ is a $(k+2)^+$-vertex, so $k=1$. 
If $\varphi(v_3)$ is a $2$-color and $v_3$ cannot be recolored with a $1$-color, then $v_3$  has two neighbors colored with each of the $1$-colors, so  $d(v_3) \ge 3= k+2$.
Otherwise, recolor $v_3$ with a $1$-color if $\varphi(v_3)$ is not already a $1$-color.
Now, we can extend $\varphi$ to a \pc{2}{k} of all of $G_k$ by coloring $u_1$ with a $2$-color and greedily coloring $u, u_2$ with $1$-colors, a contradiction.

 Now we may assume $u_1u_2$ is not an edge. 
 If at least two of $\varphi(v_1),\varphi(v_2),\varphi(v_3)$ are  $2$-colors, then we can extend $\varphi$ to a \pc{2}{k} of all of $G_k$ by greedily coloring $u,u_1,u_2$ with $1$-colors, a contradiction. 
 Hence, we may assume at most one of $\varphi(v_1),\varphi(v_2), \varphi(v_3)$ is a $2$-color.~\textbf{(1)} 

If we can extend $\varphi$ to $u$ with a $2$-color, then we can greedily color $u_1, u_2$ with $1$-colors to obtain a \pc{2}{k} of all of $G_k$, which is a contradiction. 
Therefore, all $2$-colors must appear on $(N[v_3] - u) \cup \{v_1,v_2\}$.~\textbf{(2)}

    Suppose $v_3$ is a $k^-$-vertex.
    If $\varphi(v_3)$ is a $2$-color, say $2_k$, then by~\textbf{(1)} and~\textbf{(2)}, $\varphi(v_1), \varphi(v_2)$ are both $1$-colors and all of $2_1, \ldots, 2_{k-1}$ must appear on $N(v_3)-u$. 
    Furthermore, $\{1_a, 1_b\} \subseteq N(v_3)-u$ since otherwise we can recolor $v_3$ with a $1$-color and color $u$ with $2_k$, so $v_3$ has at least $k+2$ neighbors, which is a contradiction.  
    If $\varphi(v_3)$ is a $1$-color, say $1_a$, then by~\textbf{(1)} and~\textbf{(2)}, exactly one of $\varphi(v_1), \varphi(v_2)$ is a $2$-color, say $\varphi(v_1) = 2_k$, and $\varphi(N(v_3)-u)=\{2_1, \ldots, 2_{k-1}\}$.
    If $\varphi(v_2)=1_b$, then we can extend $\varphi$ to a \pc{2}{k} of all of $G_k$ by greedily coloring $u, u_1, u_2$ with $1$-colors, a contradiction.
    Otherwise, $\varphi(v_2)=1_a$.
    Now, uncolor $v_3$, and extend $\varphi$ to a \pc{2}{k} of all of $G_k$ by greedily coloring $u, u_1, u_2, v_3$ with $1$-colors,  a contradiction.

    Suppose $v_1,v_2,v_3$ are all $(k+1)^-$-vertices.
    By~\textbf{(1)}, there are two cases to consider. 

    \textbf{Case 1:} None of $\varphi(v_1),\varphi(v_2), \varphi(v_3)$ is a $2$-color. 
    Since $d(v_3)\leq k+1$, by~\textbf{(2)}, $\varphi(N(v_3)-u)=\{2_1, \ldots, 2_{k}\}$.
    Thus, $v_3$ can be (re)colored to a $1$-color of our choice.
    Therefore, the only case where we cannot extend $\varphi$ to all of $G_k$ is when $\varphi(v_1) \neq \varphi(v_2)$.
    If $\varphi$ can be extended to $u_1$ with a $2$-color, then we can further extend $\varphi$ to a \pc{2}{k} of all of $G_k$ by greedily coloring $u_2, u, v_3$ with $1$-colors, a contradiction.
    Since $d(v_1)\leq k+1$, $\varphi(N(v_1)-u_1)=\{2_1, \ldots, 2_k\}$, so we can recolor $v_1$ to a $1$-color of our choice. 
    Therefore, we can extend $\varphi$ to a \pc{2}{k} of all of $G_k$ by greedily coloring $u_2,u,v_2,u_1,v_1$ with $1$-colors, a contradiction.

    \textbf{Case 2:} Exactly one of $\varphi(v_1),\varphi(v_2), \varphi(v_3)$ is a $2$-color. 
    By symmetry, we have the following two subcases.

    \textbf{Case 2.1:} $\varphi(v_3)$ is a $2$-color. 
    If we can recolor $v_3$ with a $1$-color, then we are done by Case 1, so $\{1_a, 1_b\}\subseteq\varphi(N(v_3) - u)$. 
    Since $d(v_3) \le k+1$, we can extend $\varphi$ to $u$ with a $2$-color, which is a contradiction to~\textbf{(2)}.

    \textbf{Case 2.2:} $\varphi(v_1)$ is a $2$-color. 
    If $\varphi(v_2) \neq \varphi(v_3)$, then we can extend $\varphi$ to a \pc{2}{k} of all of $G_k$ by greedily coloring $u_2,u,u_1$ with $1$-colors, a contradiction.
    Otherwise, $\varphi(v_2) = \varphi(v_3)$, say $1_a$. 
    If $\varphi$ can be extended to $u_2$ with a $2$-color, then we can further extend $\varphi$ to a \pc{2}{k} of all of $G_k$ by greedily coloring $u,u_1$ with $1$-colors, a contradiction.
    Otherwise, all $2$-colors appear in $\varphi(N(v_2) - u_2)$.
    However, since $d(v_2)\leq k+1$, we can recolor $v_2$ with $1_b$, and color $u_2,u,u_1$ with $1_a, 1_b, 1_a$, respectively, to obtain a \pc{2}{k} of all of $G_k$, which is a contradiction. 
\end{proof}

\begin{lemma}\label{lem:one2nb}
Let $u$ be a $3$-vertex of $G_k$ with neighbors $u_1, u_2, u_3$. 
If $u_1$ is a $2$-vertex, then $d(u_2)+d(u_3)\geq k+1$. 
In particular, $u$ has a $\ceil{\frac{k+1}{2}}^+$-neighbor. 
\end{lemma}
\begin{proof}
Suppose $d(u_2)+d(u_3)\leq k$. 
Let $v_1$ be the neighbor of $u_1$ that is not $u$. 
By the minimality of $G$, the graph $G - u_1$ has a \pc{2}{k} $\varphi$. 
Assume $\varphi(v_1)$ is a $1$-color, say $1_a$.
If $\varphi(u)=1_a$ or $\varphi(u)$ is a $2$-color, then we can extend $\varphi$ to a \pc{2}{k} of all of $G_k$ by coloring $u_1$ with $1_b$, a contradiction. 
Thus, $\varphi(u)=1_b$, and we cannot recolor $u$ with $1_a$, so $1_a\in \varphi(\{u_2, u_3\})$. 
Now, we can recolor $u$ with a $2$-color since there are at most $|N[u_2]\cup N[u_3]-\{u\}|-1\leq k-1$ $2$-colors for $u$ to avoid. 
Now, extending $\varphi$ to $u_1$ by coloring $u_1$ with $1_b$ is a \pc{2}{k} of all of $G_k$, which is a contradiction. 

Now assume $\varphi(v_1)$ is a $2$-color, say $2_k$. 
If $\varphi(u)\neq 2_k$, then we can extend $\varphi$ to a \pc{2}{k} of all of $G_k$ by greedily coloring $u_1$ with a $1$-color,  a contradiction. 
Thus, $\varphi(u)=2_k$, and we cannot recolor $u$ with a $1$-color, so $\{1_a, 1_b\}\subseteq \varphi(\{u_2,u_3\})$. 
Now, we can recolor $u$ with a $2$-color that is not $2_k$ since there are at most $|N[u_2]\cup N[u_3]-\{u\}|-2+1\leq k-1$ $2$-colors for $u$ to avoid. 
Now, extending $\varphi$ to $u_1$ by coloring $u_1$ with a $1$-color is a \pc{2}{k} of all of $G_k$, which is a contradiction. 
\end{proof}

\subsection{Discharging}\label{discharging}

Let $G_{12}$ be a counterexample to Theorem~\ref{thm:main-theorem-1} with the minimum number of vertices. 
We use the discharging method to obtain a contradiction. 
For each vertex $v$, let the initial charge $ch(v)$ be $d(v)$. 
We establish a list of discharging rules to redistribute the charge so that every vertex has final charge at least $\frac{14}{5}$. 
This contradicts that $G_{12}$ has maximum average degree less than $\frac{14}{5}$. 
The discharging rules are as below:

\begin{enumerate}[\bf {Rule }1:]
    \item\label{12R:high} Every $14^+$-vertex sends charge $\frac{4}{5}$ to each neighbor.
    \item\label{12R:mid} For $d\in\{5, \ldots, 13\}$, every $d$-vertex sends charge $\frac{2}{5}$ to each neighbor.
    \item\label{12R:4} Every $4$-vertex sends charge $\frac{2}{5}$ to each $2$-neighbor.
    \item\label{12R:3} Every $3$-vertex sends charge $\frac{2}{5}$ to each $2$-neighbor whose other neighbor is a $13^-$-vertex.
\end{enumerate}

Let $u$ be a $j$-vertex of $G_{12}$. 
If $j\geq 14$, then by Rule~\ref{12R:high}, the final charge of $u$ is at least $j - \frac{4}{5}j = \frac{1}{5}j \ge \frac{14}{5}$. 
If $j\in\{5, \ldots, 13\}$, then by Rule~\ref{12R:mid}, the final charge of $u$ is at least $j - \frac{2}{5}j = \frac{3}{5}j \ge 3$. 
If $j=4$, then $u$ has a $3^+$-neighbor by Lemma~\ref{lem:all2nb}, so by Rule~\ref{12R:4}, the final charge of $u$ is at least $4 - 3 \cdot \frac{2}{5} = \frac{14}{5}$. 
If $j=3$, then $u$ has a $3^+$-neighbor by Lemma~\ref{lem:all2nb}, so $u$ has at most two $2$-neighbors. 
If $u$ has no $2$-neighbors, then its final charge is $3$ since the charge does not change. 
If $u$ has exactly one $2$-neighbor, then by Lemma~\ref{lem:one2nb}, $u$ has a $7^+$-neighbor who sends charge at least $\frac{2}{5}$ to $u$, and $u$ sends charge at most $\frac{2}{5}$ to the $2$-neighbor of $u$. 
Therefore, the final charge of $u$ is at least $3$. 
Now assume $u$ has exactly two $2$-neighbors $u_1, u_2$, where $v_i$ is the neighbor of $u_i$ that is not $u$. 
Let the third neighbor of $u$ be $v_3$.
By Lemma~\ref{lem:two2nbs}, $v_3$ is a $13^+$-vertex. 
If $v_3$ is a $14^+$-vertex, then $v_3$ sends charge $\frac{4}{5}$ to $u$ by Rule~\ref{12R:high} and $u$ sends charge at most $2 \cdot \frac{2}{5}$ to its $2$-neighbors by Rule~\ref{12R:3}, so $u$ has final charge at least $3$. 
Otherwise, $v_3$ is a $13$-vertex, and Lemma~\ref{lem:two2nbs} further implies at least one of $v_1$ and $v_2$ is a $14^+$-vertex, say $v_1$.
Now, $v_3$ sends charge $\frac{2}{5}$ to $u$ by Rule~\ref{12R:mid} and $u$ sends charge at most $\frac{2}{5}$ to $u_2$, so the final charge of $u$ is at least $3$.  

If $j=2$, then let $u_1$ and $u_2$ be the neighbors of $u$.
If $u$ has a $14^+$-neighbor $u_i$, then $u_i$ sends charge $\frac{4}{5}$ to $u$ by Rule~\ref{12R:high} and $u$ sends no charge, so the final charge of $u$ is at least $2+\frac{4}{5}=\frac{14}{5}$.
Thus, we may assume $u_1, u_2$ are $13^-$-vertices. 
Moreover, if $u_i$ is a $2$-vertex, then $u_{3-i}$ is a $14^+$-vertex by Lemma~\ref{lem:2thread}, so we may assume each $u_i$ is a $d$-vertex where $d\in\{3, \ldots, 13\}$. 
Now, each $u_i$ sends charge $\frac{2}{5}$ to $u$ by Rules~\ref{12R:mid},~\ref{12R:4},~\ref{12R:3}, so the final charge of $u$ is  $2+\frac{4}{5}$. 
Therefore, every vertex has final charge at least $\frac{14}{5}$.

\section{List version of packing $S$-coloring}\label{list}

In this section, we extend two results of Gastineau and Togni~\cite{GT1} stating that every subcubic graph has both a \pc{1}{6} and a \pc{2}{3}. 
Note that our proof method is different from the proofs of Gastineau and Togni~\cite{GT1}. 


\begin{figure}[ht]
\begin{center}
  \includegraphics[scale=0.55]{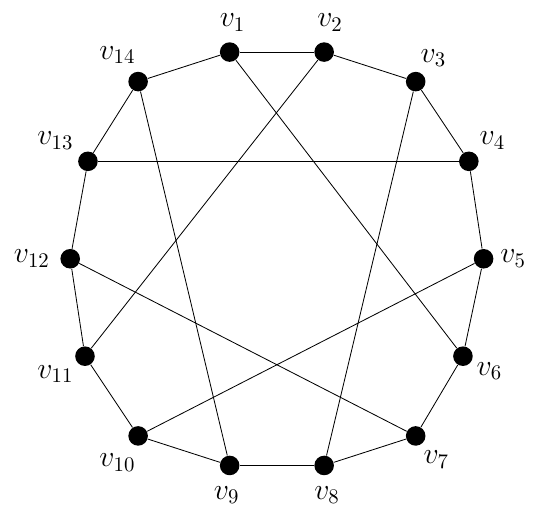} \hspace{3mm}
  \includegraphics[scale=0.55]{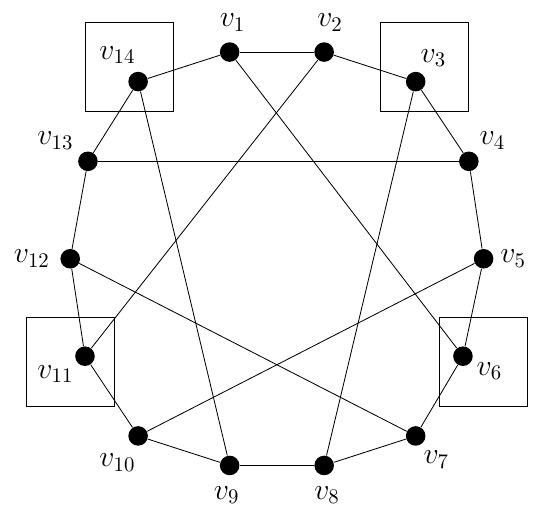} \hspace{3mm}
  \includegraphics[scale=0.55]{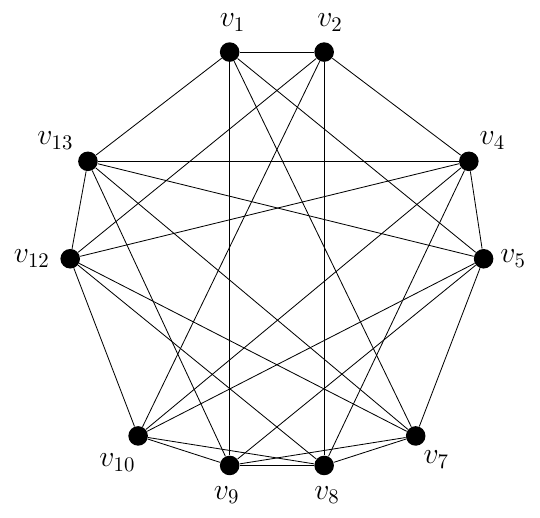} 
\caption{The Heawood graph $H$ on the left, and the graph $H'$ on the right.}\label{heawood}
\end{center}
\vspace{-8mm}
\end{figure}

We first show that the Heawood graph 
$H$ is \pch{1}{6}. 
We use the labels in \Cref{heawood}.
To see this, let $I=\{v_3, v_6, v_{11}, v_{14}\}$, and let $H'$ be the graph obtained from the square of $H$ and removing the vertices in $I$. 
Since $I$ is an independent set in $H$, for each vertex $v$ in $I$ we can use the $1$-color of $v$ on $v$. 
Since $H'$ is $5$-degenerate (realized by say, $v_4, v_5, v_7, v_8, v_9, v_{10}, v_{12}, v_{13}, v_1, v_2$), $H'$ is $6$-choosable, so $H$ is \pch{1}{6}. 

\begin{thm}\label{listthm1}
Every subcubic graph is \pch{1}{6}.
\end{thm}

\begin{proof}
Let $G$ be a connected subcubic graph and let $L$ be a  list assignment such that $|L_1(v)| = 1$ and $|L_2(v)| = 6$ for each $v \in V(G)$. 
We may assume $G$ is cubic since every connected subcubic graph is a subgraph of a connected cubic graph. 
Since we already showed the Heawood graph is \pch{1}{6}, we assume $G$ is not the Heawood graph. 
Take a maximum independent set $I$ of $G$. 
For each $v \in I$, we color $v$ with the $1$-color from $L_1(v)$. 
Let $G'$ be the graph obtained from the square of $G$ and removing the vertices in $I$. 
We remark that every vertex of $G'$ is adjacent to a vertex in $I$ in $G$. 

We claim that the maximum degree of $G'$ is at most $6$. 
Indeed, let $u$ be a vertex of $G'$, so $u\not\in I$. 
For each neighbor $u_i$ of $u$ in $G$, 
some vertex in $N_G[u_i]-\{u\}$ is in $I$. 
Thus, $d_{G'}(u)\leq 6$. 



By Vizing's Theorem~\cite{V1} (list version of Brooks's Theorem~\cite{1941Brooks}), $G'$ is $6$-choosable unless a component of $G'$ is $K_7$. 
We now show $G'$ has no  component isomorphic to $K_7$. 
Suppose $C_1$ is a component of $G'$ isomophic to $K_7$.
Let $v \in V(C_1)$ so $d_{G'}(v)=6$. 
Let $N_G(v) = \{v_1, v_2, v_3\}$, $N(v_1) = \{v, v_4, v_5\}$, $N(v_2) = \{v, v_6, v_7\}$, and $N(v_3) = \{v, v_8, v_9\}$. 

\textbf{Case 1:} Two of $v_1, v_2, v_3$ are in $V(C_1)$, say $v_1,v_2 \in V(C_1)$ and $v_3 \in I$.  Let $v_4, v_6, v_8, v_9 \in V(C_1)$ and $v_5, v_7 \in I$. 
Since $d_{G'}(v) = 6$, $|\{v_4, v_6, v_8, v_9\}| = 4$. 
First assume $v_4v_6 \in E(G)$.
Each $z\in\{v_4, v_6\}$ must have a neighbor in $I$, so $z$ cannot be adjacent to also $v_8$ or $v_9$. 
In order for $v_1$ (resp. $v_2$) to be adjacent to $v_8$ and $v_9$ in $G'$, $v_5$ (resp. $v_7$) must be adjacent to both $v_8$ and $v_9$.
Yet, $v_3$ is adjacent to neither $v_4$ nor $v_6$ in $G'$, which is a contradiction to  $C_1$ being a component of $G'$ isomorphic to $K_7$. 

Thus, we may assume $v_4v_6 \notin E(G)$, which further implies $v_4$ and $v_6$ have a common neighbor $w$. 
In order for $v_4$ (resp. $v_6$) to be adjacent to $v_2$ (resp. $v_1$) in $G'$, $v_4$ (resp. $v_6$) must be adjacent to $v_7$ (resp. $v_5$). 
In order for $z\in\{v_8, v_9\}$ to be adjacent to $v_1$ (resp. $v_2$) in $G'$, $z$ must be adjacent to a vertex in $\{v_4, v_5\}$ (resp. $\{v_6, v_7\}$).
Therefore, without loss of generality, $w=v_8$ and $v_5, v_7$ have a common neighbor $v_9$. 
Hence, $G$ is isomorphic to the Petersen graph and $I=\{v_3, v_5, v_7\}$.
Yet, the maximum size of an independent set in the Petersen graph is $4$, which contradicts the maximality of $I$. 

\textbf{Case 2:} 
One of $v_1, v_2, v_3$ is in $V(C_1)$, say $v_1 \in V(C_1)$ and $v_2, v_3 \in I$. 
Let $v_5 \in I$ and $v_4, v_6, v_7, v_8, v_9 \in V(C_1)$. 
Since $d_{G'}(v) = 6$, $|\{v_1, v_4, v_6, v_7, v_8, v_9\}| = 6$. 
Since the distance between $v_1$ and each of $v_6, v_7, v_8, v_9$ is at most $2$, $v_4$ is adjacent to two of $v_6, v_7, v_8, v_9$. 
Therefore, $v_4$ has no neighbor in $I$, which contradicts the maximality of $I$.

\textbf{Case 3:} 
None of $v_1, v_2, v_3$ is in $V(C_1)$, so $v_1, v_2, v_3\in I$. 
Since $d_{G'}(v) = 6$, $|\{v_4, v_5, v_6, v_7, v_8, v_9\}| = 6$. 
Since $v_4$ is adjacent to each of $v_6, v_7, v_8, v_9$ in $G'$, $v_4v_5 \notin E(G)$. Similarly, $v_6v_7, v_8v_9 \notin E(G)$. 
Assume  $v_5v_6 \in E(G)$.
Since $v_6$ is adjacent to $v_4, v_8, v_9$ in $G'$, without loss of generality, we may assume $v_5v_8\in E(G)$ and $v_6$ has a neighbor $w$ where $v_4, v_9\in N_G[w]$. 
If $w\not\in V(C_1)$ or $w=v_4$, then since $v_8v_9\not\in E(G)$, $v_5$ is not adjacent to $v_9$ in $G'$, a contradiction.
Otherwise, $w=v_9$.
Since $v_5$ (resp. $v_9$) is adjacent to $v_7$ in $G'$, $v_8$ (resp. $v_4$) is adjacent to $v_7$. 
Hence, $G$ is isomorphic to the Petersen graph and $I=\{v_1, v_2, v_3\}$. 
Yet, the maximum size of an independent set in the Petersen graph is $4$, which contradicts the maximality of $I$. 

Now assume $\{v_4,v_5,v_6,v_7,v_8,v_9\}$ is an independent set in $G$. 
Each of $v_4$ and $v_5$ must be connected to $v_6,v_7,v_8,v_9$ via two vertices in $I$. 
Let $N(v_4) = \{v_1, v_{10}, v_{11}\}$ and $N(v_5) = \{v_1, v_{12}, v_{13}\}$. 
If $v_{10}$ is adjacent to $v_6, v_7$, then $v_{11}$ is adjacent to $v_8, v_9$. 
Now, $v_6$ cannot be adjacent to all of $v_5, v_8, v_9$ in $G'$, a contradiction. 
Otherwise, $v_{10}$ is adjacent to $v_6, v_8$, and $v_{11}$ is adjacent to $v_7, v_9$. 
In order for $v_6$ to be adjacent $v_5, v_9$ and $v_7$ to be adjacent to $v_5, v_8$ in $G'$, up to symmetry, $v_{12}$ must be adjacent to $v_6, v_9$ and $v_{13}$ must be adjacent to  $v_7, v_8$. 
Hence $G$ is the Heawood graph, which contradicts the choice of $G$. 
%
\end{proof}

\begin{thm}\label{listthm2}
Every subcubic graph is \pch{2}{3}.    
\end{thm}

\begin{proof}
Let $G$ be a connected subcubic graph and let $L$ be a  list assignment such that $|L_1(v)| = 2$ and $|L_2(v)| = 3$ for each $v \in V(G)$. 
We may assume $G$ is cubic since every connected subcubic graph is a subgraph of a connected cubic graph. 
Use a $1$-color on as many vertices as possible, and let $I_1$ denote the set of vertices colored with a $1$-color.
Among all maximum $I_1$, let $I$ be one that minimizes the number of  components of $G':= G-I$.
For each vertex $v \in I$,  let $\varphi(v)$ denote the $1$-color used on $v$.

We claim that the maximum degree of $G'$ is at most $1$. 
Indeed, if a vertex $u$ of $G'$ has at most one neighbor $v$ in $I$, then we can add $u$ to $I$ by using a $1$-color in $L_1(u)-\{\varphi(v)\}$ on $u$ to extend $\varphi$ to $u$. 
This contradicts the maximality of $I$. 
%
%
Therefore, each component of $G'$ is either a $K_1$ or a $K_2$. 
Let $H_I$ be the graph obtained from the square of $G$ and removing the vertices in $I$.
For simplicity, we use $H$ instead of $H_I$.

\begin{claim}\label{triangle}
If $u\in I$ with neighbors $u_1, u_2, u_3\in V(G')$, then 
$u_1u_2, u_2u_3, u_3u_1 \not\in E(G)$.
Moreover, if $v\neq u$ is a neighbor of $u_i$, then $v \in I$ and all neighbors of $v$ except $u_i$ are in $I$.
\end{claim}

\begin{proof}
We first show that every neighbor of $u_i$ must be in $I$. 
If a neighbor $v$ of $u_i$ is not in $I$, then we can add $u_i$ to $I$ by removing the ($1$-)color on $u$, coloring $u_i$ with a $1$-color in $L_1(u_i)-\varphi(N(u_i))$, and recoloring $u$ with a $1$-color in $L_1(u)-\{\varphi(u_i)\}$.
This contradicts the maximality of $I$.
Since all neighbors of $I$ are not in $I$, no neighbors of $u$ are adjacent to each other. 

%

Let $v\neq u$ be a neighbor of $u_i$. 
Suppose a neighbor $v'\neq u_i$ of $v$ is not in $I$. 
From $I$, we can remove $v$ and add $u_i$ by removing the ($1$-)colors on $v$ and $u$, coloring $u_i$ with a $1$-color in $L_1(u_i)-\varphi(N(u_i))$, and recoloring $u$ with a $1$-color in $L_1(u)-\{\varphi(u_i)\}$.
This decreases the number of components of $G'$, which contradicts the choice of $I$. 
%
%
\end{proof}

By Vizing's Theorem~\cite{V1} (list version of Brooks's Theorem~\cite{1941Brooks}), $H$ is $3$-choosable if $H$ has maximum degree at most $3$, unless a component of $H$ is $K_4$. 
We now show that $H$ has maximum degree at most $3$ and   $H$ has no component isomorphic to $K_4$. 
Given a vertex $u$ of $H$, there are two cases depending on the component of $u$ in $G'$.
Note that the component is either $K_1$ or $K_2$. 

\textbf{Case 1:} 
$u$ is in a component isomorphic to $K_1$ in $G'$. 
Let $N(u) = \{u_1, u_2, u_3\}$, so $u_1, u_2, u_3 \in I$. Let $N(u_1) = \{u, u_4, u_5\}$, $N(u_2) = \{u, u_6, u_7\}$, and $N(u_3) = \{u, u_8, u_9\}$. 
If both $u_i$ and $u_{i+1}$ are in $V(G')$ for some $i\in\{4,6,8\}$, then by Claim~\ref{triangle}, $\{u_4, u_5, u_6, u_7, u_8, u_9\}-\{u_i, u_{i+1}\}\subseteq I$, so $d_H(u) = 2$. 
Otherwise at least one of $u_i, u_{i+1}$ is in $I$ for every $i\in\{4,6,8\}$. 
In every case, $d_H(u)\leq2$, except when exactly one of $u_i, u_{i+1}$ is in $I$ for every $i\in\{4,6,8\}$, $d_H(u)=3$.   
In the exceptional case, some color $1_x\in L_1(u)$ appears once  among $\varphi(u_1), \varphi(u_2), \varphi(u_3)$, say $\varphi(u_i)$. 
From $I$, we can remove $u_i$ and add $u$ by removing the (1-)color on $u_i$ and coloring $u$ with $1_x$.
This decreases the number of components of $G'$, which contradicts the choice of $I$.

\textbf{Case 2:} $u$ is in a component $uv$ isomorphic to $K_2$ in $G'$. 
Let $N(u) = \{v, u_1, u_2\}$ and $N(v) = \{u, v_1, v_2\}$, so $u_1, u_2, v_1, v_2\in I$.
By Claim~\ref{triangle}, for $z\in\{u, v\}$, every neighbor $w\not\in\{u, v\}$ of $z$ has a neighbor $w'\not\in\{u, v\}$  that is in $I$.
Thus, $d_H(u)\leq2$, except when $|\{u_1,u_2,v_1,v_2\}|=4$, $d_H(u)=3$. 
In the exceptional case, $L_1(u)=\{\varphi(u_1), \varphi(u_2)\}$ since $u\not\in I$ and $L_1(v)=\{\varphi(v_1), \varphi(v_2)\}$ since $v\not\in I$.
Without loss of generality we may assume $\varphi(u_1)\neq\varphi(v_1)$. 
From $I$, we can remove $u_1, v_1$ and add $u, v$ by removing the (1-)colors on $u_1,v_1$ and coloring $u$ and $v$ with $\varphi(u_1)$ and $\varphi(v_1)$, respectively. 
This decreases the number of components of $G'$ and contradicts the choice of $I$, unless $N[v_1]-v \subseteq I$. However, this means $d(v) \le 2$ and $H$ cannot be isomorphic to $K_4$.
\end{proof}

\section{Open Questions}
\begin{itemize}
    \item Determine $f_7(2)$, namely, the minimum $k$ for which every planar graph with girth at least $7$ has a \pc{2}{k}.
    We have shown $2\leq f_7(2)\leq 12$.

    \item Open Question~\ref{op:gk}: 
    For $k\in\{1, \ldots, 11\}$, determine $g_k$, namely, the minimum $g$ such that every planar graph with girth at least $g$ has a \pc{2}{k}.
    We know $8 \le g_1 \le 10$ and $7 \le g_k \le 8$ for $k\in\{2, \ldots, 11\}$.

    \item Determine if every subcubic planar graph is \pch{1}{5} or \pch{2}{2}.
\end{itemize}

\bibliographystyle{abbrv}
{\small
\bibliography{ref}}

\end{document}